\documentclass[a4paper,12pt]{amsart}

\usepackage[top=24mm, left=22mm, bottom=22mm, right=22mm]{geometry}
\usepackage{amsfonts, amsmath, amssymb, amsgen, amsthm, amscd}

\usepackage{newtxtext,newtxmath,mathtools}
\usepackage[utf8]{inputenc} 
\usepackage[T1]{fontenc}
\usepackage[colorlinks=true]{hyperref}

\usepackage{enumerate}
\usepackage{color}
\usepackage[all]{xy}

\newtheorem{theorem}{Theorem}[section]
\newtheorem{proposition}[theorem]{Proposition}
\newtheorem{lemma}[theorem]{Lemma}
\newtheorem{corollary}[theorem]{Corollary}

\theoremstyle{definition}
\newtheorem{remark}[theorem]{Remark}

\newtheorem{definition}[theorem]{Definition}

\newcommand{\C}[1]{\mathcal{#1}}
\newcommand{\B}[1]{\mathbb{#1}}

\newcommand{\rmod}[1]{\text{{\bf Mod}-}{#1}}
\newcommand{\xla}[1]{\xleftarrow{#1}}

\newcommand{\tl}{\triangleleft}
\newcommand{\tr}{\triangleright}
\newcommand{\CB}{\text{CB}}

\newcommand{\Span}{\text{\rm Span}}
\newcommand{\Tor}{\text{\rm Tor}}
\newcommand{\Ind}{\text{\rm Ind}}
\newcommand{\Res}{\text{\rm Res}}
\newcommand{\Ext}{\text{\rm Ext}}
\newcommand{\Hom}{\text{\rm Hom}}
\newcommand{\CoInd}{\text{\rm CoInd}}
\newcommand{\Simp}{{\Delta}}
\newcommand{\Diff}{{\Omega}}

\numberwithin{equation}{section} 

\title{A Model Categoric Equivalence for Crossed Simplicial Modules}

\author{Haydar Can Kaya}
\email{kayah17@itu.edu.tr}
\author{Atabey Kaygun}
\email{kaygun@itu.edu.tr}

\address{Istanbul Technical University, Turkey}

\begin{document}

\begin{abstract}
 We construct a model categorical equivalence between the category of simplicial vector spaces and the category of representations of a crossed simplicial group $\Delta G$ when each $G_n$ is finite and the characteristic of the ground field is 0.
\end{abstract}

\maketitle

\section*{Introduction}

Simplicial sets have a long history going back to Eilenberg and
Zilber~\cite{eilenberg1950semi}. However, the first extension of the simplicial category $\Simp$ to
include group actions compatible with the simplicial structure maps is Connes' cyclic category
$\Simp C $~\cite{connes1983cohomologie, quillen1984cyclic}.  Connes considered the action of cyclic
groups on the simplicial category to define cyclic homology in an invariant way. In this setup,
cyclic groups can be replaced with other interesting groups, such as the additive group of integers
$\B{Z}$~to construct the paracyclic category
$\Simp\B{Z}$~\cite{GetzlerJones:CyclicHomologyOfCrossedProductAlgebras}, or with symmetric groups to
construct the symmetric category~$\Simp S$~\cite{Ault:SymmetricHomology}, or dihedral groups, hyperoctahedral groups, or even braid
groups~\cite{Aboughazi:CrossedSimplicialGroups, Loday:CyclicHomology} to construct similar structures. Such structures are called
\emph{crossed simplicial groups} and are first introduced by
Fiedorowicz-Loday~\cite{fiedorowicz1991crossed} and Krasauskas~\cite{krasauskas1987skew},
independently.

It was probably Kan~\cite{kan1955abstract,kan1956abstract} who first observed that there is an intimate connection between simplicial sets and
homotopy theory of topological spaces even though the ideas
behind and the methods supporting simplicial gadgets are largely algebro-combinatorial in
nature. This point of view was later effectively encoded by Quillen in the notion of (closed) model
categories~\cite{quillen1967homotopical}. As such, the proper context for investigating a homotopy
theory of representations of crossed simplicial groups must go through endowing their category with
a suitable model structure such as~\cite{dwyer1987three}.

Our aim is to show a specific model categorical equivalence between the category of
simplicial vector spaces $\rmod{\Simp}$ and the category of representations of a crossed simplicial
group $\rmod{\Simp G}$ when $G$ is locally finite.  The prototypical example of such equivalences is
the classical Dold-Kan equivalence~\cite{DoldPuppe:Correspondence}, indeed an isomorphism, between the category of
simplicial abelian groups and differential graded abelian groups, or the Dwyer-Kan equivalence
between representations of the cyclic crossed simplicial group and duplicial
modules~\cite{dwyer1985normalizing}. The earliest example we know of a model-categorical equivalence of the desired type is the (monoidal) model-categorical (weak) equivalence between the category of differential graded modules and the category of simplicial modules by Schwede and Shipley~\cite{SchwedeShipley00,SchwedeShipley03}. The contemporary literature is now replete with such equivalences~\cite{Hinich16, chorny2023variant, richter-symmpropDold,
  shoikhet2011bialgebra, sore-coalgebraDold, SoreHermann17, Lurie:HigherAlgebra,
  dyckerhoff-categorified, peroux22, walde-homotopycoherent, peroux2023monoidal,
  truong2023operadic}.

We start by showing that the category of differential graded vector spaces $\rmod{\Diff}$ and the category of simplicial vector spaces
$\rmod{\Simp}$ are equivalent after a suitable localization in Theorem~\ref{thm:adjoint-pair}. We know that there is a
cofibrantly generated model category structure on
$\rmod{\Simp}$~\cite{jardine_2003,SchwedeShipley03}. In Section~\ref{sect:transfer1}, we show in that this model category structure is lifted from a
cofibrantly generated model structure on $\rmod{\Diff}$ using the functors implementing the equivalence given in Theorem~\ref{thm:adjoint-pair}. Thus Theorem~\ref{thm:adjoint-pair}
becomes a model categorical equivalence.  Our first main result is
Theorem~\ref{thm:HomotopyEquivalence} where we prove that when $G$ is locally finite, the categories $\rmod{\Simp}$ and
$\rmod{\Simp G}$ are equivalent after a suitable homotopy localization in the same vein as
Theorem~\ref{thm:adjoint-pair}. Then in Theorem~\ref{thm:QuillenEquivalence} we show that we can
lift the model structure we had on $\rmod{\Simp}$ to a model structure on the category of crossed
simplicial vector spaces $\rmod{\Simp G}$ using~\cite{hess2015, hess2017necessary} to give us the
correct homotopy equivalences in $\rmod{\Simp G}$. Thus the equivalence we show in
Theorem~\ref{thm:HomotopyEquivalence} is indeed a model theoretic equivalence between
$\rmod{\Simp G}$ and $\rmod{\Simp}$.

\subsection*{Notation and conventions}

We fix a ground field $\Bbbk$ of characteristic 0.  For a $\Bbbk$-vector space $V$, we will use
$V^\vee$ to denote its $\Bbbk$-vector space dual. We consider $\B{N}$ as a discrete category on the
set of objects $\B{N}$ with nothing but the identity morphisms on objects. For a small category
$\C{C}$ we use $\C{C}_{\Bbbk}$ for the algebra generated by the arrows of $\C{C}$ (called \emph{the
  categorical algebra of $\C{C}$}), or just $\C{C}$, by abuse of notation. However, for the
categorical algebra of the discrete category $\B{N}$, we use $\Bbbk_\bullet$. All modules are
assumed to be right modules unless otherwise stated. All $\Bbbk_\bullet$-modules are assumed to be
locally finite in the sense that $X_\bullet = \bigoplus_{n\in\B{N}} X_n$ where $X_n$ is defined as
$X_\bullet\cdot id_n$.

\subsection*{Acknowledgements}

We would like to thank Mehmet Akif Erdal for reading an earlier version of this manuscript and giving us valuable feedback.

\section{Preliminaries}

\subsection{Simplicial and cosimplicial modules}

We consider the skeletal category of finite well-ordered sets $\Simp$. We consider the generators
$[n+1]\xla{\partial^n_i}[n]$ and $[n]\xla{\sigma^n_i}[n+1]$ where
$[n]=\{0<\cdots<n\}$ and
\[ \partial^n_i(j) =
  \begin{cases}
    j   & \text{ if } j< i\\
    j+1 & \text{ if } j\geq i
  \end{cases}
  \qquad
  \sigma^n_i(j) =
  \begin{cases}
    j   & \text{ if } j\leq i\\
    j-1 & \text{ if } j> i
  \end{cases}
\]
There are non-trivial relations among these generators
\[ \partial^{n+1}_i\partial^n_j = \partial^{n+1}_{j+1}\partial^n_i, \qquad
  \sigma^n_j\sigma^{n+1}_i = \sigma^n_i\sigma^{n+1}_{j+1}, \qquad
  \sigma^n_j\partial^n_i =
  \begin{cases}
    \partial^{n-1}_i\sigma^{n-1}_{j-1} & \text{ if } i<j\\
    id & \text{ if } i=j, j+1\\
    \partial^{n-1}_{i-1}\sigma^{n-1}_j & \text{ if } i>j+1
  \end{cases}
\]
whenever $i\leq j$. See, for example~\cite{May:SimplicialObjects}.  With this notation, right and
left $\Simp $-modules respectively are simplicial and cosimplicial $\Bbbk$-modules.

Note that, the interaction relations between $\sigma^n_j$'s and $\partial^n_i$'s can be rewritten
as
\[ \partial^{n-1}_i \sigma^{n-1}_j =
  \begin{cases}
    \sigma^n_{j+1}\partial^n_i & \text{ if } i\leq j\\
    \sigma^n_j \partial^n_{i+1} & \text{ if } i>j
  \end{cases}
\]
Now, we see that any morphism in $\Simp$ of the form $[m]\xla{\varphi}[n]$ can be written uniquely
as a monomial of the form
\begin{equation}
  \label{eq:monomials1}
   \sigma^m_{\ell_m}\cdots \sigma^r_{\ell_r} \partial^r_{i_r}\cdots \partial^n_{i_n}
\end{equation}
where $i_r>\cdots > i_n$ and $\ell_r >\cdots > \ell_m$ and $i_r\neq \ell_r, \ell_r+1$.

\subsection{Crossed simplicial modules} \cite[1.1. Definition]{fiedorowicz1991crossed} We are interested in any sequence of groups $\{G_n\}_{n\in\B{N}}$ that allows us to define a new category $\Simp G$ that extends the simplicial category $\Simp$ so that 
\begin{enumerate}[(a)]
    \item the category $\Simp G$ and $\Simp$ share the same set of objects,
    \item the automorphism group of each object $[n]$ in $\Simp G$ is $G_n^{op}$, and
    \item any morphism in $\Simp G$ can be uniquely written in the form of a composite $\phi \cdot g$ of an element $\phi \in \Hom_{\Simp}([m],[n])$ and $g \in G_m^{op}$.
\end{enumerate}
In other words, in such an extension we require that morphisms to be represented by monomials of the form
\begin{equation}
  \label{eq:monomials3}
  \sigma^m_{\ell_m}\cdots \sigma^r_{\ell_r} \partial^r_{i_r}\cdots \partial^n_{i_n} g
\end{equation}
where $g\in G_n$, $i_r>\cdots > i_n$ and $\ell_r >\cdots > \ell_m$ and $i_r\neq \ell_r,
\ell_r+1$. The extension category $\Simp G$ is called \emph{a crossed simplicial group}. The
notable examples include Connes' cyclic category $\Simp C$ constructed out of cyclic groups
$\{\B{Z}/(n+1)\B{Z}\}_{n\in\B{Z}}$, the symmetric category $\Simp S$ constructed out of symmetric
groups $\{S_{n+1}\}_{n\in\B{N}}$, and the paracyclic category $\Simp \B{Z}$ constructed out of
infinite cyclic group $\{\B{Z}\}_{n\in\B{N}}$.

\subsection{Differential graded modules}

Let us define
\begin{equation}
    d_{i,n} = \sum_{j=i}^n (-1)^j \partial^n_j
\end{equation}
Notice that $d_{i,n+1}d_{i,n} = 0$ for every $n$ and $0\leq i\leq n+1$.

Now, define $\Diff$ as the unital subalgebra of $\Simp$ generated by $1_n$ and $d_{0,n}$ for every
$n$. With this notation, $\Diff$-modules are differential graded $\Bbbk$-modules where
differentials decrease the degree by 1 for right $\Diff$-modules, increase the degree by 1 for left
$\Diff$-modules.

\subsection{A change of basis}

\begin{lemma}
  The categorical algebra $\Simp$ has a basis that consists of monomials of the form
  \begin{equation}
    \label{eq:monomials2}
    \sigma^m_{\ell_n}\cdots \sigma^r_{\ell_r}d_{j_r,r}\cdots d_{j_n,n}
  \end{equation}
  again with $j_r>\cdots > j_n$ and $\ell_r >\cdots > \ell_m$.
\end{lemma}

\begin{proof}
  We show that one can replace monomials of the form
  $\sigma^m_{\ell_m}\cdots \sigma^r_{\ell_r}\partial^r_{i_r}\cdots \partial^n_{i_r}$ with monomials
  of the form $\sigma^m_{\ell_m}\cdots \sigma^r_{\ell_r}d_{j_r,r}\cdots d_{j_n,n}$ where
  $j_r>\cdots>j_n$, and $\ell_r >\cdots > \ell_m$ and $j_r\neq i_r,i_r+1$. Let us prove this by
  induction on the length of the monomials in $\partial$'s. The statement is clear for monomials of
  length 1. So, assume we have the statement for all monomials of $\partial$'s of length $N$. Take
  a monomial of the form
  $\sigma^m_{\ell_m}\cdots \sigma^r_{\ell_r}\partial^r_{i_r}\cdots \partial^n_{i_n}$ where $r-n=N$
  and $i_r>\cdots>i_n$ and $\ell_m >\cdots > \ell_r$. Since we can replace the initial part of the
  monomials in $\partial$'s with the monomials of the correct type, we get a sum of the form
  \begin{align*}
    \sigma^m_{\ell_m}\cdots \sigma^r_{\ell_r}\partial^r_{i_r}\cdots \partial^n_{i_n}
    = & \sum _{j_r>\cdots>j_{n+1}>i_n} \lambda_{j_m,\ldots,j_{n+1}}\sigma^m_{\ell_m}\cdots \sigma^r_{\ell_r} d_{j_r,r}\cdots d_{j_{n+1},n+1}\partial^n_{i_n}\\
    = & \sum _{j_r>\cdots>j_{n+1}>i_n} (-1)^{i_n}\lambda_{j_m,\ldots,j_{n+1}}\sigma^m_{\ell_m}\cdots \sigma^r_{\ell_r} d_{j_r,r}\cdots d_{j_{n+1},n+1}(d_{i_n,n}-d_{i_n+1,n})
  \end{align*}
  The only part of the sum which does not conform to our statement comes from monomials of the form
  \begin{equation*}
    \sigma^m_{\ell_m}\cdots \sigma^r_{\ell_r}d_{j_r,r}\cdots d_{j_{n+1},n+1} d_{i_n+1,n}
  \end{equation*}
  where we still have $j_r>\cdots>j_{n+1}>i_n$.  As long as $j_{n+1}>i_n+1$ the sum still conforms
  to the statement. The only case we must resolve is when $j_{n+1} = i_n+1$ in which case we use
  $d_{i,a+1}d_{i,a} = 0$. This proves that the set of monomials given in
  Equation~\eqref{eq:monomials2} is a spanning set. However, the set of monomials given in
  Equation~\eqref{eq:monomials1} and the set of monomials given in Equation~\eqref{eq:monomials2}
  have the same finite size since they are indexed by the same of set of sequences. Thus the set of
  monomials given in Equation~\eqref{eq:monomials2} forms a basis.
\end{proof}

\begin{proposition}\label{prop:free}
  $\Simp$ is a free right $\Diff$-module.
\end{proposition}

\begin{proof}
  Since $\Simp$ has a basis by the monomials of the form given in Equation~\eqref{eq:monomials2} we
  can write $\Simp$ as a product of the form
  \[ \bigoplus_{n\geq 0} \Span_{\Bbbk}(id_n,\sigma^m_{\ell_m}\cdots \sigma^r_{\ell_r}
    d_{i_r,r}\cdots d_{i_{n+1},n+1}\mid i_r>\cdots> i_{n+1}, \ell_r>\cdots>\ell_m)\otimes
    \Span_\Bbbk(1_n, d_{0,n}) \] The result follows.
\end{proof}

\begin{proposition}
    There is a split-epimorphism of unital algebras of the form $\eta\colon \Simp\to \Diff$ defined by
    \[ \eta(\sigma^m_{\ell_m}\cdots \sigma^r_{\ell_r}d_{j_r,r}\cdots d_{j_n,n}) = 
       \begin{cases}
          d_{0,n} & \text{ if $r<n+1$, $m=n$, and $j_n=0$}\\
          0 & \text{ otherwise.}
       \end{cases}
    \]
    whose splitting is given by the inclusion $\Diff\to\Simp$.
\end{proposition}

\subsection{Induction, coinduction and restriction}

Given two algebras $A$ and $B$, and an algebra morphism $f\colon A\to B$, one can consider $B$ as
an $A$-module via $f$.  There are three important functors that we are going to use repeatedly in
this text. These are
\begin{enumerate}[(i)]
\item The induction functor $\Ind_{A\to B}\colon \rmod{A}\to\rmod{B}$ given by
  $(\ \cdot\ )\otimes_{A}B $,
\item The coinduction functor $\CoInd_{A\to B}\colon \rmod{A}\to\rmod{B}$ given by
  $\Hom_{A}(B,\ \cdot\ )$, and
\item The restriction functor $\Res_{A\to B}\colon \rmod{A}\to\rmod{B}$ given by considering a
  $B$-module as an $A$-module via the morphism $f\colon A\to B$.
\end{enumerate}

The morphisms we consider below are the inclusion of algebras $\Diff\to \Simp$, and the epimorphism
of algebras $\Simp\to \Diff$ we defined above. There is another epimorphism (but of bimodules of a
certain type) $\Simp G\to \Upsilon$ that we will consider, but it is defined in
Definition~\ref{defn:upsilon} below.

\section{Homology and Homotopy of Crossed Simplicial Modules}

\subsection{Point objects}

We define the trivial left $\Simp$-module $\Bbbk_\bullet$ as $\Bbbk_n = \Bbbk$ where all generators of $\Simp$ act by identity.  Similarly, we define the trivial left $\Diff$-module $\Bbbk[0]$ as the $\Diff$-module which is $0$ everywhere except degree 0 where we have $\Bbbk$, and the generators of the subring $\Diff$ act by 0 except $1_0$ which acts as the identity on $\Bbbk$ in degree 0. Notice that we have $\Bbbk_\bullet = \Ind_{\Diff\to\Simp}\Bbbk[0] = \CoInd_{\Diff\to\Simp}\Bbbk[0]$. 

\begin{proposition}\label{prop:quasi-equivalence}
  We have
  $\Res_{\Diff\to\Simp} \Bbbk_\bullet$ is quasi-isomorphic to $\Bbbk[0]$. Thus we have
  \[ \Tor^{\Simp}_*(X_\bullet,\Bbbk_\bullet)\cong 
     \Tor^{\Diff}_*(\Res_{\Diff\to\Simp}X_\bullet,\Bbbk[0])\cong 
     H_*(\Res_{\Diff\to\Simp}X_\bullet) 
  \] and
  \[ \Ext_{\Simp}^*(X_\bullet,\Bbbk_\bullet^\vee) 
     \cong \Tor^{\Simp}_*(X_\bullet,\Bbbk_\bullet)^\vee
     \cong H_*(\Res_{\Diff\to\Simp}X_\bullet)^\vee
  \]
\end{proposition}

\begin{proof}
  Let $\Simp_n$ be the left $\Simp$-ideal generated by $1_n$.  We define differentials
  $d_n\colon \Simp_n\to\Simp_{n-1}$ via right multiplication by $d_{0,n-1}$. Since $\Simp$ is a
  free $\Diff$-module by Proposition~\ref{prop:free}, if $\Psi\in ker(d_n)$ then
  $\Psi = \Psi' d_{0,n}$. Thus, we show that this chain complex is exact except for degree 0. We
  then observe that $\Simp_0/im(d_0)$ is $\Bbbk_\bullet$. 
\end{proof}

\subsection{Quasi-isomorphisms and homotopy equivalences}

Recall that a morphism of $\Diff$-modules $f_*\colon X_*\to Y_*$ is called \emph{a
  quasi-isomorphism} if the induced maps $H_n(f_*)\colon H_n(X_*)\to H_n(Y_*)$ in homology are
isomorphisms.  On the simplicial side, given a simplicial module $X_\bullet$, one can define
\emph{combinatorial simplicial group} $\pi_n(X_\bullet)$ as follows. First, we define
\[ Z_n = \{ x\in X_n\mid \partial_i(x)=0, \text{ for all } i=0,\ldots,n\} \] 
Then we write an equivalence relation $\sim$ by letting $x\sim x'$ in $Z_n$ if there is an element $y\in X_{n+1}$ such that
\[ \partial_i(y) =
  \begin{cases}
    0 & \text{ if } 0\leq i< n\\
    x & \text{ if } i=n\\
    x' & \text{ if } i=n+1
  \end{cases}\] And finally, we define $\pi_n(X_\bullet) = Z_n/\!\!\sim$.  Now, a morphism
$f_\bullet\colon X_\bullet\to Y_\bullet$ of $\Simp$-modules as a \emph{homotopy equivalence} if the induced morphisms on the homotopy groups $\pi_n(f_\bullet)\colon \pi_n(X_\bullet)\to \pi_n(Y_\bullet)$ are isomorphisms.

\begin{theorem}\label{thm:TorExt}
  Assume $f_\bullet\colon X_\bullet\to Y_\bullet$ is a morphism of simplicial $\Bbbk$-vector spaces. Then the following are equivalent:
  \begin{enumerate}
      \item $f_\bullet$ is a homotopy equivalence.
      \item $\Res_{\Diff\to\Simp}(f_\bullet)\colon \Res_{\Diff\to\Simp}(X_\bullet)\to \Res_{\Diff\to\Simp}(Y_\bullet)$  is a quasi-isomorphism.
      \item $\Tor^{\Simp}_n(f_\bullet,\Bbbk_\bullet)\colon \Tor^{\Simp}_n(X_\bullet,\Bbbk_\bullet) \to \Tor^{\Simp}_n(Y_\bullet,\Bbbk_\bullet)$ is an isomorphism for every $n\geq 0$.
      \item $\Ext_{\Simp}^n(f_\bullet,\Bbbk_\bullet^\vee)\colon \Ext_{\Simp}^n(Y_\bullet,\Bbbk_\bullet^\vee) \to \Ext_{\Simp}^n(X_\bullet,\Bbbk_\bullet^\vee)$ is an isomorphism for every $n\geq 0$.
  \end{enumerate}
\end{theorem}

\begin{proof}
  By Proposition~\ref{prop:quasi-equivalence} we know that the induced morphisms are isomorphisms
  iff $\Res_{\Diff\to\Simp}(f_\bullet)$ is a quasi-isomorphism. By~\cite[Thm
  8.3.8]{Weibel:HomologicalAlgebra} we know that
  $\pi_n(X_\bullet)\cong H_n(\Res_{\Diff\to\Simp}(X_\bullet))$. This means $f_\bullet$ is a
  homotopy equivalence iff $\Res_{\Diff\to\Simp}(f_\bullet)$ is a quasi-isomorphism.  The result
  follows.
\end{proof}

\begin{corollary}
  Let $\Bbbk[0]$ be the right $\Simp$-module $\Bbbk$ concentrated at degree 0, and where all generators act by 0 except for $1_0$.  Then the natural embedding $\Bbbk[0]\to \Bbbk_\bullet$ is a homotopy equivalence.
\end{corollary}

\subsection{Homotopy equivalence for crossed simplicial modules}

Assume $\Simp G$ is a crossed simplicial group and consider the embedding $\Simp\to \Simp G$ which
comes with its own induction and restriction functors $\Res_{\Simp\to\Simp G}$ and
$\Ind_{\Simp\to\Simp G}$. 

We define $\Bbbk_\bullet$ to be the left $\Simp G$-module where each face map, degeneracy map, and
the group element act by identity. Now, a morphism of crossed simplicial $\Simp G$-modules
$f_\bullet\colon X_\bullet\to Y_\bullet$ is called a homotopy equivalence if the induced maps on
the torsion groups
\[ \Tor^{\Simp G}_n(f_\bullet,\Bbbk_\bullet)\colon \Tor^{\Simp G}_n(X_\bullet,\Bbbk_\bullet)\to \Tor^{\Simp G}_n(Y_\bullet,\Bbbk_\bullet) 
\] are isomorphisms. Instead of the torsion groups, if we wanted to use extension groups we would require \[ \Ext^{\Simp G}_n(f_\bullet,\Bbbk^\vee_\bullet)\colon \Ext^{\Simp G}_n(Y_\bullet,\Bbbk^\vee_\bullet)\to \Ext^{\Simp G}_n(X_\bullet,\Bbbk^\vee_\bullet) \] be isomorphisms. But, these conditions are equivalent.

In order to be able to say anything about homotopy equivalences of crossed simplicial modules, we
would need a resolution of the point object $\Bbbk_\bullet$ as a $\Simp G$-module.  The following
is a slight modification of \cite[Theorem 5.12 and Proposition 5.13]{Kayg11}.

\begin{proposition}\label{prop:EquivariantHomotopyEquivalence}
  Assume that each $G_n$ is finite, i.e. $G\subseteq \Simp G$ is a semi-simple subalgebra. Then
  the absolute derived functors $\Tor^{\Simp G}$ and $\Ext_{\Simp G}$ are isomorphic to their
  relative counterparts $\Tor^{(\Simp G|G)}$ and $\Ext_{(\Simp G|G)}$.  Moreover, morphism of
  $\Simp G$-modules $f_\bullet\colon X_\bullet\to Y_\bullet$ is a homotopy equivalence iff
  \begin{equation}
    \label{eq:crossed-simplicial-homotopy}
    \Tor^{\Simp}_n(\Res_{\Simp\to\Simp G} f_\bullet,\Bbbk_\bullet)_G\colon \Tor^{\Simp}_n(\Res_{\Simp\to\Simp G} X_\bullet,\Bbbk_\bullet)_G\to \Tor^{\Simp}_n(\Res_{\Simp\to\Simp G} Y_\bullet,\Bbbk_\bullet)_G
  \end{equation}
  is an isomorphism for every $n\geq 0$ where $(U_\bullet)_G$ denotes level-wise $G$-coinvariants of a $G$-module $U_\bullet$.
\end{proposition}

\begin{proof}
  Consider the two-sided bar complex $\CB_*(X_\bullet,\Simp G,\Bbbk_\bullet)$ 
  \[ X_\bullet\otimes_{\Bbbk_\bullet}\underbrace{\Simp
      G\otimes_{\Bbbk_\bullet}\cdots\otimes_{\Bbbk_\bullet}\Simp G
    }_{n\text{-times}}\otimes_{\Bbbk_\bullet} \Bbbk_\bullet \] that calculates
  $\Tor^{\Simp G}_*(X_\bullet,\Bbbk_\bullet)$ and its relative counterpart
  $\CB_*(\Simp G,\Simp G|G,\Bbbk_\bullet)$
  \[ X_\bullet\otimes_{G}\underbrace{\Simp G\otimes_{G}\cdots\otimes_{G}\Simp G
    }_{n\text{-times}}\otimes_{G} \Bbbk_\bullet \] that calculates the $G$-relative Tor groups
  $\Tor^{\Simp G|G}_*(X_\bullet,\Bbbk_\bullet)$. Note that since each $G_n$ is finite and $\Bbbk$
  has characteristic 0, the absolute Tor groups and the relative Tor groups are isomorphic by
  \cite[Proposition 2.5]{Kayg11}.  Since we have a basis of the form~\eqref{eq:monomials3}, we can
  reduce the relative complex to
  \[ X_\bullet\otimes_{G}\left(\underbrace{\Simp\otimes_{\Bbbk_\bullet}\cdots\otimes_{\Bbbk_\bullet}\Simp}_{n\text{-times}}\otimes_{\Bbbk_\bullet} \Bbbk_\bullet\right)
  \]
  Let us notice two things: (i) the right hand side component comes from the $\Simp$-resolution of
  $\Bbbk_\bullet$ with a diagonal action of $G$, and (ii) $X_\bullet\otimes_G Y_\bullet$ can be
  written as a coinvariant module as
  $\Bbbk_\bullet \otimes_G (X^{op}_\bullet\otimes_{\Bbbk_\bullet} Y_\bullet)$ where the left
  $G$-action on $X_\bullet$ is converted to a right action of $G$ on $X_\bullet^{op}$ via
  $g\tr x:= x\tl g^{-1}$. Thus we get
  \[ \Bbbk_\bullet\otimes_G
    \left(X^{op}_\bullet\otimes_{\Bbbk_\bullet}\underbrace{\Simp\otimes_{\Bbbk_\bullet}\cdots\otimes_{\Bbbk_\bullet}\Simp}_{n\text{-times}}\otimes_{\Bbbk_\bullet}
      \Bbbk_\bullet\right)\] The right tensor component calculates
  $\Tor^{\Simp}_*(\Res^{\Simp G}_\Simp X^{op}_\bullet,\Bbbk_\bullet)=\Tor^{\Simp}_*(\Res^{\Simp
    G}_\Simp X_\bullet,\Bbbk_\bullet)$. The result follows.
\end{proof}

\begin{corollary}\label{cor:CoinvariantHomotopy}
  Let $f_\bullet\colon X_\bullet\to Y_\bullet$ be a morphism of $\Simp G$-modules with the same
  assumptions in Proposition~\ref{prop:EquivariantHomotopyEquivalence}. Then $f_\bullet$ is a
  homotopy equivalence iff the induced morphisms
  $\Ext_{\Simp}^n(\Res_{\Simp\to\Simp G} f_\bullet,\Bbbk_\bullet^\vee)_G$ and
  $\Ext_{\Simp}^n(\Res_{\Simp\to\Simp G} f_\bullet,\Bbbk_\bullet^\vee)^G$ are isomorphisms for every
  $n\geq 0$.
\end{corollary}

\begin{remark}
  From this point on we will assume that $G$ is semisimple in $\Simp G$.
\end{remark}

\section{The Dold-Kan Equivalence Re-imagined}

\subsection{The Moore functor}

Let $X_\bullet$ be simplicial $\Bbbk$-module, i.e. a right $\Simp$-module. For every $n\in\B{N}$
let us define $N_n(X_\bullet)$ as
\begin{equation}
    N_n(X_\bullet) = \{ x\in X_n\mid x\partial_i = 0 \text{ for } 1\leq i\leq n \}
\end{equation}
together with the differentials $\partial_0^n$ acting on the right.

\begin{lemma}
  For every right $\Simp$-module $X_\bullet$, $N_*(X_\bullet)$ is exactly the vector subspace of
  $X_\bullet$ on which $ker(\eta)$ acts by 0. In other words, $N_*$ is the functor
  $\CoInd_{\Simp\to\Diff}\colon \rmod{\Simp}\to \rmod{\Diff}$.
\end{lemma}

\begin{proof}
  Let us consider the functor $\CoInd_{\Simp\to\Diff}\colon \rmod{\Simp}\to\rmod{\Diff}$ on the
  objects:
  \begin{align}
    \CoInd_{\Simp\to\Diff}X_\bullet 
    = & \Hom_{\Diff}(\Diff,X_\bullet) \\
    = & \{x\in X_n\mid x d_{i,n-1} = 0 \text{ for } n\geq 1, \text{ and } 1\leq i\}\\
    = & \{x\in X_n\mid x \partial^{n-1}_i = 0 \text{ for } n\geq 1, \text{ and } 1\leq i\} 
  \end{align}
  Then when we reduce the $\Diff$-action to an $\Simp$-action, we declare all $d_{i,n}$'s act by 0
  for $i\geq 1$. Thus $\CoInd_{\Simp\to\Diff}X_\bullet = N_*(X_\bullet)$.
\end{proof}

\begin{proposition}[{\cite[8.3.8]{Weibel:HomologicalAlgebra}}{\cite[Theorem III 2.1]{GoerssJardine:SimplicialHomotopy}}]\label{prop:equivalence}
  The functors $\CoInd_{\Simp\to\Diff}$ and $\Res_{\Diff\to\Simp}$ are naturally isomorphic after
  we localize $\rmod{\Simp}$ with respect to homotopy equivalences and $\rmod{\Diff}$ with respect
  to quasi-isomorphisms.
 \end{proposition}

 \begin{proof}
   The statement is equivalent to the fact that the natural inclusion
   $N_*(X_\bullet) \to \Res_{\Diff\to\Simp} X_\bullet$ is a quasi-isomorphism for every simplicial
   module $X_\bullet$.
 \end{proof}

\subsection{The classical Dold-Kan equivalence via coinduction and restriction}

 \begin{theorem}\label{thm:adjoint-pair}
   The categories $\rmod{\Simp}$ and $\rmod{\Diff}$ are equivalent via adjoint the functors
   $\Res_{\Simp\to\Diff}$ and $\CoInd_{\Simp\to\Diff}$
   \[ \xymatrix{
       \rmod{\Simp} \ar@/^{2ex}/[rrr]^{\CoInd_{\Simp \to\Diff}} &&&
       \ar@/^{2ex}/[lll]^{\Res_{\Simp\to\Diff}} \rmod{\Diff} 
     }\]
   after we localize these categories with respect to homotopy equivalences and quasi-isomorphisms,
   respectively.
 \end{theorem}

 \begin{proof}
   Let us take a differential graded module $Y_*$ and consider $\Res_{\Simp\to\Diff}Y_*$.  This is
   the same graded module $Y_*$ considered as a simplicial $\Bbbk$-module on which every $d_{i,n}$
   act by 0 for $i>0$, or equivalently every face map $\partial^n_i$ act by 0 for $i>0$.  Now, if
   we consider $\CoInd_{\Simp\to\Diff}\Res_{\Simp\to\Diff}Y_*$, we get the graded vector subspace
   of $Y_* = \Res^{\Diff}_{\Simp}Y_*$ given by the intersection of the kernels of the face maps
   $\partial^n_i$ for $i>0$. This is again $Y_*$ with the same differentials. Thus
   $\CoInd_{\Simp\to\Diff}\Res_{\Simp\to\Diff}$ is the identity functor. On the other hand, take a
   simplicial $\Bbbk$-module $X_\bullet$, and consider the composition
   $\Res_{\Simp\to\Diff}\CoInd_{\Simp\to\Diff}X_\bullet$. This is the normalized Moore complex
   $N_*(X_\bullet)$ considered as a simplicial module where every face map $\partial^n_i$ acts by
   $0$ for $i>0$. However, since the homology groups of
   $\Res_{\Simp\to\Diff}\CoInd_{\Simp\to\Diff}X_\bullet$, $N_*(X_\bullet)$, and $X_\bullet$ are all
   isomorphic, the natural embedding
   $\Res_{\Simp\to\Diff}\CoInd_{\Simp\to\Diff}X_\bullet \to X_\bullet$ is a homotopy equivalence.
 \end{proof}

\subsection{An equivalence for crossed simplicial groups}~

We consider the derived functors of the functor $\Hom_{\Delta G}(\ \cdot\ ,\Bbbk^\vee_\bullet)$ for
the homotopy equivalences of crossed simplicial modules. However, this functor can be factored into
a composition of two functors as
\[ \Hom_{\Simp G}(\ \cdot\ ,\Bbbk^\vee_\bullet) \cong \Hom_{\Simp}(\Res_{\Simp\to\Simp G}\ \cdot\
  ,\Bbbk^\vee_\bullet)^G \] where we apply the restriction functor, and after calculating the
derived functors, then we take the $G$-invariants. We would like to reverse the order of the
application of these functors. 

The main difficulty in reversing the order of applications of these functors comes from the fact
that the subspace of $G$-invariants of a $\Simp G$-module $X_\bullet$ is not necessarily a
$\Simp$-submodule. Similarly, when in defining $G$-coinvariants, we need to divide $X_\bullet$ by a
vector subspace of elements spanned by elements of the form $x\tl (g-1)$ that do not necessarily
form a $\Simp$-submodule.

\begin{definition}\label{defn:upsilon}
  Let $\epsilon_G$ be the left $\Simp G$- and right $\Simp$-submodule of $\Simp G$ generated by
  elements of the form $g-1$ with $g\in G_n$ for every $n\geq 0$.  Now, define $\Upsilon$ as the
  $\Simp G$-$\Simp$-bimodule quotient $\Simp G/\epsilon_G$.
\end{definition}

With this definition at hand, one can now define another adjoint pair
\[ \xymatrix{ \rmod{\Simp G} \ar@/^2ex/[rrr]^{\Ind_{\Simp G\to\Upsilon}} &&&
    \ar@/^2ex/[lll]^{\CoInd_{\Simp G\to\Upsilon}} \rmod{\Simp} }\] where
$\Ind_{\Simp G\to \Upsilon} = (\ \cdot\ )\otimes_{\Simp G}\Upsilon$ and
$\CoInd_{\Simp G\to\Upsilon} = \Hom_{\Simp}(\Upsilon, \ \cdot\ )$.
Now, we observe that
\begin{equation}
  \label{eq:observation}
   X_\bullet\otimes_{\Simp G}\Bbbk_\bullet
  \cong X_\bullet\otimes_{\Simp G}\Upsilon\otimes_{\Simp}\Bbbk_\bullet
  = (X_\bullet)_G\otimes_{\Simp}\Bbbk_\bullet
\end{equation}

\begin{theorem}\label{thm:HomotopyEquivalence}
  The categories $\rmod{\Simp G}$ and $\rmod{\Simp}$ are equivalent via the functors
  $\Ind_{\Simp G\to\Upsilon}$ and $\Ind_{\Simp\to\Simp G}$
  \[ \xymatrix{ \rmod{\Simp G} \ar@/^{2ex}/[rrr]^{\Ind_{\Simp G\to\Upsilon}} &&&
      \ar@/^{2ex}/[lll]^{\Ind_{\Simp \to\Simp G}} \rmod{\Simp} }\] after we localize underlying
  categories with respect to the corresponding homotopy equivalences.
\end{theorem}

\begin{proof}
  We first observe that, for a $\Simp G$-module $X_\bullet$ and a $\Simp$-module $Y_\bullet$ we
  have
  \[ \Ind_{\Simp G\to\Upsilon}(X_\bullet) = X_\bullet\otimes_{\Simp G}\Upsilon \cong (X_\bullet)_G \] and
  \[ \Ind_{\Simp\to\Simp G}(Y_\bullet) = Y\otimes_{\Simp}\Simp G \cong Y\otimes_{\Bbbk_\bullet}
    G \] These identities indicate that we have isomorphisms of functors of the form
  \[ \Ind_{\Simp G\to\Upsilon}\Ind_{\Simp\to\Simp G} \cong \left(\ \cdot\ \otimes_{\Bbbk_\bullet}
      G\right)_G\cong id_{\rmod{\Simp}} \] On the opposite side we have
  \[ \Ind_{\Simp\to\Simp G}\Ind_{\Simp G\to \Upsilon} \cong (\ \cdot\ )_G\otimes_{\Bbbk_\bullet}
    G \] which implies there is a natural transformation of the form
  $\eta_\bullet\colon id_{\rmod{\Simp G}} \to (\ \cdot\ )_G\otimes_{\Bbbk_\bullet} G$ which, in
  turn, induces isomorphisms in induced maps on the torsion groups as
  $\Tor^{\Simp G}_*(\eta_\bullet,\Bbbk_\bullet)$ since
  \[ \left((\ \cdot\ )_G\otimes_{\Bbbk_\bullet} G \right)_G \cong (\ \cdot\ )_G \]
  The result follows.
\end{proof}

\begin{corollary}\label{cor:main-result}
We have the following equivalences
\[ \xymatrix{
    \rmod{\Simp G} \ar@/^2ex/[rr]^{\Ind_{\Simp G\to\Upsilon}} 
    && \ar@/^2ex/[ll]^{\Ind_{\Simp\to\Simp G}} \rmod{\Simp} \ar@/^2ex/[rr]^{\CoInd_{\Simp\to\Diff}} 
    && \ar@/^2ex/[ll]^{\Res_{\Simp\to\Diff}} \rmod{\Diff}
}\]
after localizing each category with respect to the corresponding homotopy equivalences or quasi-isomorphisms.
\end{corollary}

\section{Model Categorical Reformulation}\label{sect:Quillen}

\subsection{A cofibrantly generated Quillen model category structure on $\rmod{\Diff}$}~

We start by describing the standard model structure on the category $\rmod{\Diff}$ of dg-$\Bbbk$-vector spaces. To define a model structure, we need to determine fibrations, cofibrations, and weak equivalences that satisfy the axioms of the model structure. 

\begin{enumerate}[(i)]

\item A \emph{cofibration} in the category of dg-$\Bbbk$-vector spaces is a monomorphism $f_*\colon X_*\to Y_*$ of dg-$\Bbbk$-vector spaces.

\item A \emph{weak equivalence} of dg-$\Bbbk$-vector spaces is a morphism $f_* \colon X_* \to Y_*$ of dg-$\Bbbk$-vector spaces that induces an isomorphism $H_n(f_*)\colon H_n(X_*)\to H_n(Y_*)$ on homology groups.

\item Now, fibrations of dg-$\Bbbk$-vector spaces are completely determined by the class of weak equivalences and cofibrations. A morphism $f_*\colon X_*\to Y_*$ of dg-$\Bbbk$-vector spaces is called a \emph{fibration} if $f_*$ has the right lifting property with respect to all trivial cofibrations.

\end{enumerate}

\subsection{Transferring the model category structure to $\rmod{\Simp}$}~

We consider the algebra epimorphism $\Simp \to \Diff$, and the three functors
$\Ind_{\Simp\to\Diff}$, $\CoInd_{\Simp\to\Diff}$, and $\Res_{\Simp\to\Diff}$ associated with this
epimorphism. The coinduction functor $\CoInd_{\Simp\to\Diff}$ is given by a Hom as
$\Hom_{\Simp}(\ \cdot\ ,\Diff)$.  The restriction functor $\Res_{\Simp\to\Diff}$, on the other
hand, is given by a tensor as $(\ \cdot\ )\otimes_{\Diff}\Simp$. As such we have an adjoint pair of
functors $\left(\Res_{\Simp\to\Diff},\CoInd_{\Simp\to\Diff}\right)$. We would like to lift the
model category structure on $\rmod{\Diff}$ to a model category structure on $\rmod{\Simp}$.  We can
then lift the model structure along the right adjoint $\CoInd_{\Simp\to\Diff}$ using~\cite[Theorems
11.3.1 and 11.3.2]{hirschhorn2009model} or~\cite[Thm.7.44]{heuts_simplicial_2022}.

Note that since we deal with $\Bbbk$-vector spaces, all monomorphisms are split. Thus the lifted
cofibrations are still the class of monomorphisms. On the other hand, the lifted weak equivalences
are those morphisms $f_\bullet\colon X_\bullet\to Y_\bullet$ which are sent to quasi-isomorphisms
under the coinduction functor $\CoInd_{\Simp\to\Diff}$ which is quasi-isomorphic to the restriction
functor $\Res_{\Diff\to\Simp}$. Thus the lifted weak equivalences are exactly the class of homotopy
equivalence of simplicial vector spaces.

The class of fibrations, cofibrations, and weak equivalences we defined above yield a cofibrantly
generated Quillen model category structure on $\rmod{\Simp}$ the category of simplicial
$\Bbbk$-vector spaces as defined in~\cite{jardine_2003,SchwedeShipley03}. Now, the result below
follows directly from Theorem~\ref{thm:adjoint-pair}.

\begin{theorem}\label{thm:DoldKan}
  The Dold-Kan equivalence is a Quillen model categorical equivalence.
\end{theorem}

\subsection{Transferring the model category structure to $\rmod{\Simp G}$}\label{sect:transfer1}~

\begin{theorem}\label{thm:QuillenEquivalence}
 The model categories $\rmod{\Simp G}$ and $\rmod{\Simp}$ are equivalent.   
\end{theorem}
\begin{proof}
Observe that the epimorphism $\Simp G\to \Upsilon$ gives us a pair of adjoint functors
\[ \xymatrix{ \rmod{\Simp G} \ar@/^{2ex}/[rrr]^{\Ind_{\Simp G\to\Upsilon}} &&&
    \ar@/^{2ex}/[lll]^{\CoInd_{\Simp G\to\Upsilon}} \rmod{\Simp} }\] Note that the left adjoint
functor $\Ind_{\Simp G\to\Upsilon} = (\ \cdot\ )\otimes_{\Simp G}\Upsilon$ is the right exact
functor $(\ \cdot\ )_G$ that calculates the $G$-coinvariants of a $\Simp G$-module as a
$\Simp$-module.  We already defined homotopy equivalences in $\rmod{\Simp G}$ using
$(\ \cdot\ )_G$.  Now, we lift the remaining model category structure on $\rmod{\Simp}$ to a model
category structure on $\rmod{\Simp G}$ along the left adjoint $\Ind_{\Simp G\to\Upsilon}$
using~\cite{hess2015, hess2017necessary}. Since we already had a homotopy equivalence given by
Theorem~\ref{thm:HomotopyEquivalence}, we get a model categorical equivalence.
\end{proof}

\begin{corollary}
  The equivalences given in Corollary~\ref{cor:main-result} are equivalences of Quillen model
  categories.
\end{corollary}

\bibliographystyle{siam}
\bibliography{bibliography}

\end{document}